\documentclass[10pt]{article}

\usepackage[T1]{fontenc}
\usepackage[utf8]{inputenc}

\usepackage[top=25mm, bottom=25mm, left=25mm, right=25mm, a4paper]{geometry}

\usepackage{graphicx}
\usepackage{multirow}
\usepackage{amsmath,amssymb,amsfonts}
\usepackage{amsthm}
\usepackage{mathrsfs}
\usepackage[title]{appendix}
\usepackage{xcolor}
\usepackage{textcomp}
\usepackage{manyfoot}
\usepackage{booktabs}
\usepackage{algorithm}
\usepackage{algorithmicx}
\usepackage{algpseudocode}
\usepackage{listings}
\usepackage[none]{hyphenat}

\usepackage{hyperref}

\numberwithin{equation}{section}

\theoremstyle{plain}
\newtheorem{theorem}{Theorem}
\newtheorem{proposition}[theorem]{Proposition}
\newtheorem{lemma}[theorem]{Lemma}

\theoremstyle{definition}

\newtheorem{example}{Example}

\theoremstyle{remark}
\newtheorem{remark}{Remark}

\makeatletter
\renewcommand{\maketitle}{%
  \begin{center}
    {\Large\bfseries \@title \par}
    \vskip 1.5em
    {\normalsize \@author \par}
    \vskip 1em
  \end{center}
}
\makeatother

\newcommand{\keywords}[1]{%
  \par\medskip
  \noindent\textbf{Keywords.} #1
}

\newcommand{\subjclass}[2][]{%
  \par\medskip
  \noindent\textbf{2020 Mathematics Subject Classification.} #2
}

\begin{document}

\sloppy

\title{Multiplicity Results for the Nonlinear Weighted Laplacian Equation on the Sierpi\'nski Gasket}

\author{Micha{\l} Be{\l}dzi\'nski,\qquad  Krzysztof Jelito}

\date{}

\maketitle

\begin{center}
\small
\textsc{Micha{\l} Be{\l}dzi\'nski}\\
Institute of Mathematics, Lodz University of Technology,\\
al. Politechniki 8, 93-590 {\L}\'od\'z, Poland\\
\texttt{michal.beldzinski@p.lodz.pl}

\vspace{0.8em}

\textsc{Krzysztof Jelito}\\
Institute of Mathematics, Lodz University of Technology,\\
al. Politechniki 8, 93-590 {\L}\'od\'z, Poland\\
Department of Mathematics, Silesian University of Technology,\\
Kaszubska 23, 44-100 Gliwice, Poland\\
\texttt{krzysztof.jelito@dokt.p.lodz.pl}\\
\texttt{Krzysztof.Jelito@polsl.pl}

\vspace{0.8em}

The authors contributed equally to this work.
\end{center}

\begin{abstract}
In this paper we consider the Dirichlet boundary value problem driven by the weighted Laplacian equation and considered on the Sierpi\'nski gasket. Under suitable growth conditions imposed on the nonlinear term, we provide existence and multiplicity results employing the direct method, mountain pass theorem and some related multiplicity result.
\end{abstract}



\section{Introduction}
We study a Dirichlet boundary value problem involving a weighted Laplacian on the
Sierpi\'nski gasket $V$, namely
\begin{equation}
    \label{equ:Main_inclusion}
    \begin{cases}
        -\Delta_a u(\boldsymbol{x}) \ni f\big(\boldsymbol{x},u(\boldsymbol{x})\big)
        & \text{for } \boldsymbol{x}\in V\setminus V_0,\\[2mm]
        u(\boldsymbol{x})=0
        & \text{for } \boldsymbol{x}\in V_0,
    \end{cases}
\end{equation}
where $f\colon V\times\mathbb{R}\to\mathbb{R}$ is an $L^1$--Carath\'eodory function satisfying suitable growth assumptions specified below, function $a$ is bounded and strictly positive. $V_0$ denotes the intrinsic boundary of $V$. The passage from equality to inclusion is necessitated by the fact that, the operator $-\Delta_a$ may be multivalued. A detailed explanation of this phenomenon will be discussed in Section \ref{subsec:Construction_of_Delta_a}.

Differential equations on fractal domains have attracted considerable attention in recent decades. We refer to \cite{Barnsley,Kigami2001,Strichartz2006} for background material on analysis on fractals and related elliptic problems.

Falconer and Hu \cite{Falconer1999} established the existence of solutions to the problem
\begin{equation*}
    \begin{cases}
        \Delta u + a(\boldsymbol{x})u
        = f\big(\boldsymbol{x},u(\boldsymbol{x})\big)
        & \text{for $\mu$-a.e. } \boldsymbol{x}\in V\setminus V_0,\\
        u(\boldsymbol{x})=0
        & \text{for } \boldsymbol{x}\in V_0,
    \end{cases}
\end{equation*}
where $a$ is integrable and the nonlinearity $f$ satisfies suitable growth
conditions near zero and at infinity. Their approach relied on variational methods,
in particular the Mountain Pass and Saddle Point theorems.

Subsequently, Molica Bisci and R\u{a}dulescu \cite{MolicaBisciRadulescu2015}, as well
as Molica Bisci, Bonanno, and R\u{a}dulescu \cite{BonannoBisciRadulescu2012}, proved
existence and multiplicity results for various generalizations of the problem
originally studied by Falconer and Hu.

In \cite{Galewski2019,Galewski2019b}, Galewski investigated the boundary value
problem
\begin{equation*}
    \begin{cases}
        \Delta u(\boldsymbol{x}) + a(\boldsymbol{x})u(\boldsymbol{x})
        = f\big(\boldsymbol{x},u(\boldsymbol{x}),w(\boldsymbol{x})\big)
        + g(\boldsymbol{x})
        & \text{for $\mu$-a.e. } \boldsymbol{x}\in V\setminus V_0,\\
        u(\boldsymbol{x})=0
        & \text{for } \boldsymbol{x}\in V_0,
    \end{cases}
\end{equation*}
with $a,g\in L^1(V)$, $w\in L^2(V)$, and a nonlinearity $f$ satisfying appropriate growth assumptions. In \cite{Galewski2019} the case $g\equiv0$ was considered and the dependence of solutions on parameters was analyzed via the Mountain Pass Theorem, while in \cite{Galewski2019b} monotonicity methods were employed to establish existence and uniqueness results in the fractal setting.

More recently, Priyadarshi and Sahu \cite{SahuPriyadarshi2021} studied systems of equations involving the $p$--Laplacian on fractal domains. 

It seems that boundary value problems involving weighted Laplacians on fractals have not yet been addressed in the literature. The definition of the weighted Laplacian proposed in the present paper allows us to work within the Hilbert space $H_0^1(V)$, which is particularly convenient for variational methods.

Our main tool in establishing the existence of nontrivial solutions is a set of abstract results from~\cite{Bereanu2012}. These theorems combine the direct minimization method with the Mountain Pass Lemma. To the best of the authors' knowledge, they have not previously been applied to equations on the Sierpi\'nski gasket $V$, even in the case of the standard Laplacian $-\Delta$. In particular, the present work appears to be the first to employ them in the study of equations driven by the weighted Laplacian considered here.

The paper is organized as follows. In Section~\ref{sec:Preliminaries}, we establish basic properties of Szulkin-type functionals together with the methods applied. Then, we recall the standard construction of the Sierpi\'nski gasket via an iterated function system, following \cite{Barnsley}. Next, we introduce the Laplacian $-\Delta$ in the sense of Kigami \cite{GibbonsRajStrichartz,Kigami2001} and subsequently define the weighted Laplacian $-\Delta_a$, emphasizing its relation to the classical Laplacian on $V$. Then, we define a Szulkin-type functional that plays the role of the Euler action functional associated with \eqref{equ:Main_inclusion}. We impose assumptions on the nonlinearity $f$ and derive a series of lemmas establishing the key properties of the corresponding nonlinear term. This leads to the formulation of the final Szulkin-type functional, whose properties are further investigated. Finally, in Section~\ref{sec: Main result} we apply these methods to problem \eqref{equ:Main_inclusion}. 

\section{Preliminaries}
\label{sec:Preliminaries}

\subsection{Variational tools}

Throughout this subsection we assume that
\begin{equation}
    \label{ass:E}
    \tag{$E$}
    \boxed{
    (E,\|\cdot\|)\ \text{is a real and reflexive Banach space.}
    }
\end{equation}
The norm in the dual space $E^*$ is denoted by $\|\cdot\|_{E^*}$. Moreover, for every $r>0$ we set
\[
B_r := \{u\in E : \|u\|<r\}.
\]
On the space $E$ we consider a Szulkin-type functional of the form
\begin{equation}
    \label{equ:Szulkin-type_functional}
    I = \psi + \Phi,
\end{equation}
where
\begin{equation}
    \label{ass:psi}
    \tag{$\psi$}
    \boxed{
    \psi \colon E \to \mathbb{R}\ \text{is convex and continuous.}
    }
\end{equation}
It is common to assume that the functional $\psi$ is lower semicontinuous. In the present setting, however, where $\psi$ takes only finite values, this assumption is in fact equivalent to continuity; see \cite[Corollary~4.1]{Galewski-book}. The same applies to weak lower semicontinuity, see \cite[Theorem~1.4]{Figueiredo}. Recall that, for every $u\in E$, the functional $\psi$ admits a subdifferential given by
\begin{equation*}
    \partial \psi(u) = \left\{\eta\in E^* : \psi(v)-\psi(u) \ge \langle \eta, v-u\rangle \text{ for all } v\in E\right\}.
\end{equation*}
Concerning $\Phi$, we shall assume that
\begin{equation}
    \label{ass:Phi}
    \tag{$\Phi$}
    \boxed{
    \begin{tabular}{l}
        $\Phi \colon E \to \mathbb{R}$ is of class $C^1$. Moreover, $\Phi'$ is weak-to-strong continuous, i.e.\\
        \qquad $u_n \rightharpoonup u$ in $E$ implies $\Phi'(u_n) \to \Phi'(u)$ in $E^*$. 
    \end{tabular}
    }
\end{equation}
Following \cite{Szulkin}, we call $u\in E$ a \emph{critical point of $I$} if $-\Phi'(u)\in \partial\psi(u)$, or equivalently, if
\[
\psi(v)-\psi(u) + \langle \Phi'(u), v-u\rangle \ge 0
\quad \text{for all } v\in E.
\]

We say that $I$ satisfies the \emph{Palais--Smale condition} if every sequence $(u_n)\subset E$ such that:
\begin{itemize}
    \item the sequence $\bigl(I(u_n)\bigr)$ is bounded,
    \item there exists a sequence $(\varepsilon_n)\subset \mathbb{R}$ with $\varepsilon_n\to 0$ and
    \begin{equation}
        \label{equ:almost_critical}
        \psi(v)-\psi(u_n)+\langle \Phi'(u_n), v-u_n\rangle
        \ge -\varepsilon_n \|v-u_n\|
        \quad \text{for all } v\in E,
    \end{equation}
\end{itemize}
admits a convergent subsequence. Sequences satisfying the above conditions are called \emph{Palais--Smale sequences} for $I$.

We impose the following additional assumption on $\psi$:
\begin{equation}
    \label{ass:psi_rho}
    \tag{$\psi_\rho$}
    \boxed{
    \begin{tabular}{l}
        there exists an increasing function $\rho\colon[0,\infty)\to[0,\infty)$ with $\rho(0)=0$ s.t.\\
        \qquad
        $\psi(v)\ge \psi(u)+\langle \eta, v-u\rangle
        +\rho(\|u-v\|)\|u-v\|$\\ 
        for all $u,v\in E$ and every $\eta\in \partial\psi(u)$.
    \end{tabular}
    }
\end{equation}

It was shown in \cite[Lemma 3.1]{Pietrusiak} that condition \eqref{ass:psi_rho} is equivalent to the $\rho$-uniform monotonicity of $\partial\psi$ and to the so-called $\rho$-uniform convexity of $\psi$. Both notions are discussed in detail in \cite{Pietrusiak}. Here we restrict ourselves to assumption \eqref{ass:psi_rho}, which allows us to prove the following counterpart of \cite[Theorem 21]{Galewski}.

\begin{lemma}
    \label{lem:every_(PS)_is_bounded=>I_(PS)}
    Assume that \eqref{ass:E}, \eqref{ass:psi}, \eqref{ass:psi_rho}, and \eqref{ass:Phi} hold. Then the functional $I$ defined by \eqref{equ:Szulkin-type_functional} satisfies the Palais--Smale condition if and only if every Palais--Smale sequence possesses a bounded subsequence.
\end{lemma}

\begin{proof}
    The necessity is obvious. We show sufficiency. As observed in \cite[Proposition 1.2]{Szulkin}, condition \eqref{equ:almost_critical} is equivalent to the existence of a sequence $(\eta_n)\subset E^*$ such that $\eta_n\to 0$ and
    \[
    \eta_n - \Phi'(u_n) \in \partial\psi(u_n).
    \]
    Let $(u_n)$ be a bounded Palais--Smale sequence. Passing to a subsequence if necessary (not relabeled), we may assume that $u_n\rightharpoonup u$ in $E$. Choosing $v=u$ in \eqref{equ:almost_critical}, and $v=u_n$ in \eqref{ass:psi_rho} with an arbitrary $\eta\in \partial\psi(u)$, we obtain
    \begin{equation}
        \label{equ:prop:every}
        \rho(\|u_n-u\|)\|u_n-u\|
        \le \varepsilon_n\|u-u_n\|
        + \langle \Phi'(u_n), u-u_n\rangle
        + \langle \eta, u-u_n\rangle .
    \end{equation}
    Since $u_n\rightharpoonup u$, the sequence $(\|u-u_n\|)$ is bounded and hence $\varepsilon_n\|u-u_n\|\to 0$. By assumption \eqref{ass:Phi}, we have $\Phi'(u_n)\to \Phi'(u)$ in $E^*$, which implies
    \[
    \bigl|\langle \Phi'(u_n), u-u_n\rangle\bigr|
    \le \|\Phi'(u_n)-\Phi'(u)\|_{E^*}\|u-u_n\|
    + |\langle \Phi'(u), u-u_n\rangle|
    \to 0.
    \]
    Moreover, $\langle \eta, u-u_n\rangle \to 0$. Therefore, \eqref{equ:prop:every} yields $\|u_n-u\|\to 0$, proving the claim.
\end{proof}

\begin{lemma}
    \label{lem:boundedness_of_I}
    Assume that \eqref{ass:E}, \eqref{ass:psi}, \eqref{ass:psi_rho}, and \eqref{ass:Phi} hold, and define $I$ by \eqref{equ:Szulkin-type_functional}. Then $I$ is bounded from below on bounded subsets of $E$.
\end{lemma}

\begin{proof}
    We show that both $\psi$ and $\Phi$ are bounded from below on bounded sets. For $\psi$ this follows from \cite[Lemma 4.2]{Pietrusiak}. To verify boundedness of $\Phi$, recall (see~\cite{Galewski-book}) that
    \[
    \Phi(u)=\int_0^1 \langle \Phi'(tu),u\rangle\,dt
    \quad \text{for all } u\in E.
    \]
    Fix $r>0$ and set
    \[
    M := \sup_{u\in B_r}\|\Phi'(u)\|_{E^*}.
    \]
    To show that $M$ is finite let $(u_n)\subset B_r$ be such that $\|\Phi'(u_n)\|_{E^*}\to M$. Passing to a subsequence, we may assume that $u_n\rightharpoonup u$ for some $u\in B_r$. By \eqref{ass:Phi}, we have $M=\|\Phi'(u)\|_{E^*}<\infty$. Consequently,
    \[
    |\Phi(u)|
    \le \int_0^1 \|\Phi'(tu)\|_{E^*}\|u\|\,dt
    \le Mr
    \quad \text{for all } u\in B_r,
    \]
    which completes the proof.
\end{proof}

\begin{proposition}
    \label{prop:I_coercive=>(PS)}
    Assume that \eqref{ass:E}, \eqref{ass:psi}, \eqref{ass:psi_rho}, and \eqref{ass:Phi} hold, and define $I$ by \eqref{equ:Szulkin-type_functional}. If $I$ is \emph{coercive}, i.e.
    \[
    \lim_{\|u\|\to\infty} I(u)=\infty,
    \]
    then $I$ satisfies the Palais--Smale condition. Moreover, $I$ attains its global minimum: there exists $u^*\in E$ such that
    \[
    I(u^*)=\min_{u\in E} I(u).
    \]
\end{proposition}

\begin{proof}
    Let $(u_n)$ be a Palais--Smale sequence for $I$. By coercivity, boundedness of $\bigl(I(u_n)\bigr)$ implies boundedness of $(u_n)$. Lemma~\ref{lem:every_(PS)_is_bounded=>I_(PS)} then shows that $I$ satisfies the Palais--Smale condition. Moreover, coercivity yields the existence of $R>0$ such that
    \[
    I(0)<\inf_{u\in E\setminus B_R} I(u)
    \quad \text{and} \quad
    I(0)<\inf_{u\in \partial B_R} I(u).
    \]
    The conclusion now follows from \cite[Proposition~1]{Bereanu2012}.
\end{proof}

The following result links \cite[Theorem~3.2]{Szulkin} with \cite[Proposition~2]{Bereanu2012}, both of which rely on the mountain pass geometry. In particular, it guarantees the existence of two critical points: one corresponding to a local minimum of $I$ on a sufficiently small ball, and another given by the saddle point provided by the Mountain Pass Lemma.

\begin{proposition}
    \label{prop:Mawhin2}
    Assume that \eqref{ass:E}, \eqref{ass:psi}, \eqref{ass:psi_rho}, and \eqref{ass:Phi} hold, and define $I$ by \eqref{equ:Szulkin-type_functional}. Suppose further that $I$ satisfies the Palais--Smale condition and that there exist $r>0$ and $e\in E\setminus \overline{B_r}$ such that
    \begin{equation}
        \label{equ:Mountain_Pass_geometry_for_I}
        I(0) < \inf_{u\in \partial B_r} I(u)
        \quad \text{and} \quad
        I(e)\le I(0).
    \end{equation}
    Then $I$ admits at least one nonzero critical point. If, in addition,
    \begin{equation}
        \label{equ:condition_for_third_critical_point}
        \inf_{u\in \overline{B_r}} I(u) < I(0),
    \end{equation}
    then $I$ possesses at least two nonzero critical points.
\end{proposition}

\begin{proof}
    Since $I$ satisfies the Palais--Smale condition, Szulkin’s version of the Mountain Pass Lemma applies; see \cite[Theorem~3.2]{Szulkin}. This yields the existence of a nonzero critical point. If both \eqref{equ:Mountain_Pass_geometry_for_I} and \eqref{equ:condition_for_third_critical_point} hold, then Lemma~\ref{lem:boundedness_of_I} allows us to apply \cite[Proposition~2]{Bereanu2012}, which completes the proof.
\end{proof}

\subsection{Sierpi\'nski gasket}

Let $\boldsymbol{p}_1$, $\boldsymbol{p}_2$, and $\boldsymbol{p}_3$ be the vertices of an arbitrary equilateral triangle in $\mathbb{R}^2$. Each point $\boldsymbol{p}_i$, $i\in\{1,2,3\}$, is the unique fixed point of the affine contraction $\boldsymbol{F}_i\colon \mathbb{R}^2\to\mathbb{R}^2$ defined by
\[
\boldsymbol{F}_i(\boldsymbol{x})=\frac{\boldsymbol{x}+\boldsymbol{p}_i}{2}
\quad \text{for all } \boldsymbol{x}\in\mathbb{R}^2.
\]
Hence, $\mathbb{R}^2$ equipped with the Euclidean norm, together with the mappings $\boldsymbol{F}_1$, $\boldsymbol{F}_2$, and $\boldsymbol{F}_3$, constitutes an iterated function system; see \cite[Section~3.7]{Barnsley}. Let $\mathcal{K}(\mathbb{R}^2)$ denote the family of all compact subsets of $\mathbb{R}^2$, endowed with the Hausdorff metric
\[
h(X,Y)
=
\max\Bigl\{
\max_{\boldsymbol{x}\in X}\min_{\boldsymbol{y}\in Y}|\boldsymbol{x}-\boldsymbol{y}|,
\max_{\boldsymbol{y}\in Y}\min_{\boldsymbol{x}\in X}|\boldsymbol{x}-\boldsymbol{y}|
\Bigr\}
\quad\text{for all } X,Y\in\mathcal{K}(\mathbb{R}^2).
\]
Define the mapping $\mathcal{F}\colon \mathcal{K}(\mathbb{R}^2)\to\mathcal{K}(\mathbb{R}^2)$ by
\[
\mathcal{F}(X)=\boldsymbol{F}_1(X)\cup \boldsymbol{F}_2(X)\cup \boldsymbol{F}_3(X)\quad \text{for every } X\in\mathcal{K}(\mathbb{R}^2).
\]
Then $\mathcal{F}$ is a contraction on $\bigl(\mathcal{K}(\mathbb{R}^2),h\bigr)$ and therefore admits a unique fixed point, known as the \emph{Sierpi\'nski gasket}, which we denote by $V$. Setting
\[
V_0=\{\boldsymbol{p}_1,\boldsymbol{p}_2,\boldsymbol{p}_3\},
\]
and defining recursively
\[
V_{n+1}=\mathcal{F}(V_n) \quad\text{for every } n\in\mathbb{N}_0,
\]
we obtain $h(V_n,V)\to 0$ as $n\to\infty$. This construction of $V$ was used to define the Laplacian in \cite{Kigami1989} and \cite{strichartzwong2004}. As in \cite{Falconer1999}, we equip $V$ with the normalized Hausdorff measure $\mu$, so that $\mu(V)=1$.

We introduce two standard notational conventions concerning points in $V$. For simplicity, we identify ordered $n$-tuples with entries in $\{1,2,3\}$ with words of length $n$ over the same alphabet. Accordingly, we set ${\Sigma_n=\{1,2,3\}^n}$ and $\Sigma_\infty=\bigcup_{n\in\mathbb{N}_+}\Sigma_n$. For $\omega=(\omega_1,\dots,\omega_n)\in\Sigma_n$ we define
\[
\boldsymbol{F}_\omega
=
\boldsymbol{F}_{\omega_1}\circ\cdots\circ \boldsymbol{F}_{\omega_n},
\]
and
\[
\boldsymbol{p}_\omega
=
\boldsymbol{F}_{\omega_1\ldots\omega_{n-1}}(\boldsymbol{p}_{\omega_n}).
\]
Consequently,
\[
    V_{n-1}=\{\boldsymbol{p}_\omega : \omega\in\Sigma_n\} \quad \text{for every } n\in\mathbb{N}_+.
\]

We now introduce the Laplace operator on $V$, following \cite{Kigami1989}. The starting point is the observation that for every $u\in C(V)$ one has
\begin{equation}
    \label{equ:fundamental_inequality}
    \sum_{1\le i<j\le 3}\bigl|u(\boldsymbol{p}_i)-u(\boldsymbol{p}_j)\bigr|^2
    \le \tfrac{5}{3}\sum_{k=1}^3\sum_{1\le i<j\le 3}
    \bigl|u(\boldsymbol{p}_{ki})-u(\boldsymbol{p}_{kj})\bigr|^2.
\end{equation}
This inequality is fundamental in the construction of the operator $\Delta$ on $V$. Informally, it states that the energy determined by the values of $u$ on $V_0$ does not exceed the energy determined on $V_1$, provided the latter is rescaled by the factor $\tfrac{5}{3}$. As a consequence, one obtains
\begin{equation}
    \label{equ:condition_E^(n)_is_increasing}
    \sum_{1\le i<j\le 3}\bigl|u(\boldsymbol{p}_{\omega i})-u(\boldsymbol{p}_{\omega j})\bigr|^2
    \le \tfrac{5}{3}\sum_{k=1}^3\sum_{1\le i<j\le 3}
    \bigl|u(\boldsymbol{p}_{\omega ki})-u(\boldsymbol{p}_{\omega kj})\bigr|^2,
\end{equation}
for every $\omega\in\Sigma_\infty$ and arbitrary $u\in C(V)$, see \cite[Section~1.3]{Strichartz2006} for details. This leads to the definition of the energy functional
\begin{equation}
    \label{equ:functional_E}
    \mathcal{E}(u) = \lim_{n\to\infty} \left(\tfrac{5}{3}\right)^n \sum_{\omega\in\Sigma_n} \sum_{1\le i<j\le 3} \bigl|u(\boldsymbol{p}_{\omega i})-u(\boldsymbol{p}_{\omega j})\bigr|^2.
\end{equation}
The limit (finite or infinite) exists, since the sequence is nondecreasing by \eqref{equ:condition_E^(n)_is_increasing}. Note that in defining $\mathcal{E}(u)$ we only use the values of $u$ on
\[
V_\infty=\bigcup_{n\in\mathbb{N}_+}V_n
=
\{\boldsymbol{p}_\omega : \omega\in\Sigma_\infty\}.
\]
It is therefore natural to restrict the definition of $\mathcal{E}$ to continuous functions. We define
\[
H^1(V)=\{u\in C(V): \mathcal{E}(u)<\infty\},
\]
and equip this space with the norm
\[
\|u\|_{\mathcal{E}}
=
\left(\mathcal{E}(u)+\int_V |u(\boldsymbol{x})|^2\,d\mu(\boldsymbol{x})\right)^{1/2}.
\]
As an analogue of the space of functions vanishing on the boundary, we introduce
\[
H_0^1(V)
=
\{u\in H^1(V): u(\boldsymbol{x})=0 \text{ for all }\boldsymbol{x}\in V_0\},
\]
which, endowed with the norm
\[
\|u\|_{H_0^1(V)}=\sqrt{\mathcal{E}(u)},
\]
is a Hilbert space; see \cite{Hu2004}. Within this framework, the Laplace operator is defined in the weak sense as $-\Delta\colon H_0^1(V)\to \bigl(H_0^1(V)\bigr)^*$ given by $-\Delta = \tfrac{1}{2} \mathcal{E}'$.
\par
This definition coincides with Kigami’s weak Laplacian on the Sierpi\'nski gasket \cite{Kigami2001}; see also \cite{Strichartz2006}. It is consistent with the classical definition on a bounded open set $\Omega\subset\mathbb{R}^n$, where $-\Delta=\operatorname{div}\nabla$ is realized as the G\^ateaux derivative of
\[
H_0^1(\Omega)\ni v\longmapsto \tfrac12\int_\Omega |\nabla v(\boldsymbol{x})|^2\,d\boldsymbol{x}.
\]
It can be shown (see \cite{Falconer1999}) that, as in the classical case, solutions to
\[
\begin{cases}
-\Delta u(\boldsymbol{x})=f(\boldsymbol{x}), & \boldsymbol{x}\in V\setminus V_0,\\
u(\boldsymbol{x})=0, & \boldsymbol{x}\in V_0,
\end{cases}
\]
with $f\in L^2(V)$ are precisely the critical points of the functional
\[
H_0^1(V)\ni v\longmapsto
\tfrac12 \mathcal{E}(v)
-
\int_V f(\boldsymbol{x})v(\boldsymbol{x})\,d\mu(\boldsymbol{x}).
\]

Another analogy with the classical setting is provided by the following Sobolev-type embedding.

\begin{lemma}[{\cite[Proposition~4.4]{Fukushima1992}}]
    \label{lemma:embedding}
    The embedding $H^1(V)\hookrightarrow C(V)$ is compact. Moreover,
    \[
    \|u\|_{\infty}\le 9\,\|u\|_{H_0^1(V)} \quad \text{for all } u\in H_0^1(V).
    \]
\end{lemma}

Since $C(V)$ embeds continuously into $L^2(V)$, it follows that $H_0^1(V)$ embeds compactly into $L^2(V)$. Let $\lambda_1$ denote the first eigenvalue of the Laplace operator with Dirichlet boundary conditions. Its reciprocal provides the optimal constant in the Poincar\'e inequality
\begin{equation}
    \label{equ:Poincare_inequality}
    \|u\|_{L^2}^2
    \le \lambda_1^{-1}\|u\|_{H_0^1}^2
    \quad \text{for all } u\in H_0^1(V).
\end{equation}
Unlike the one-dimensional case, no explicit formula for $\lambda_1$ is known. Numerical approximations based on finite element methods yield $\lambda_1\approx 16.816$; see \cite[Table~5.1]{GibbonsRajStrichartz}.

\subsection{Construction of a weighted Laplacian on the Sierpi{\'n}ski gasket}
\label{subsec:Construction_of_Delta_a}

We begin by recalling the definition of the weighted Laplacian $-\Delta_a$ in the classical setting, that is, on an open and bounded subset $\Omega$ of Euclidean space. This serves as a point of reference for the construction that follows. Building on the resulting analogies--while emphasizing the essential differences--we then proceed to define the operator $-\Delta_a$ on the Sierpiński gasket.
\par
Let us recall that the Laplace operator $-\Delta$ in its weak formulation--namely, as a mapping from $H^1_0(\Omega)$ into the dual space $H^{-1}(\Omega)$--is defined as the G{\^a}teaux derivative of the functional
\begin{equation*}\label{equ:potential_of_Delta}
    H^1_0(\Omega) \ni v \longmapsto 
    \tfrac{1}{2}\int_\Omega |\nabla v(\boldsymbol{x})|^2\, d\boldsymbol{x}.
\end{equation*}
Given a function $a\colon \Omega \to \mathbb{R}$ that is measurable, bounded, and bounded away from zero, we obtain the corresponding potential for the operator $-\Delta_a$ in the form
\begin{equation*}\label{equ:potential_of_Delta_a}
    H^1_0(\Omega) \ni v \longmapsto 
    \tfrac{1}{2}\int_\Omega a(\boldsymbol{x})\,|\nabla v(\boldsymbol{x})|^2\, d\boldsymbol{x}.
\end{equation*}
In the case of the Sierpiński gasket, the potential of the Laplacian is given by \(\tfrac{1}{2}\mathcal{E}\), where the energy form \(\mathcal{E}\) is defined by \eqref{equ:functional_E}. It is important to note that the definition of \(\mathcal{E}(u)\) relies exclusively on the values of \(u\) on the set \(V_\infty\), which has measure zero. Consequently, one cannot unambiguously define the operator \(-\Delta_a\) using an expression of the form
\begin{equation}
    \label{equ:functional_E_a_for_continuous}
    \mathcal{E}_a(u) = \lim_{n\to \infty} \left(\tfrac{5}{3}\right)^{n} \sum_{\omega \in \Sigma_n} \sum_{1 \leq i < j \leq 3} \tfrac{a(\boldsymbol{p}_{\omega i}) + a(\boldsymbol{p}_{\omega j})}{2}\big|u(\boldsymbol{p}_{\omega i})-u(\boldsymbol{p}_{\omega j})\big|^2.
\end{equation}
for an arbitrary function \(a\in L^\infty(V)\). However, it turns out that the above definition can indeed be employed in the case of continuous functions, since \(V_\infty\) is dense in \(V\). Additional properties of the functional \(\mathcal{E}_a\) are described in the following 
\begin{lemma}
    \label{lem:E_a_for_continuous_a}
    Take a continuous function $a\colon V\longrightarrow (0,\infty)$. Then, for every $u\in H^1_0(V)$, there exists a limit
    \begin{equation}
        \label{equ:functional_E_a_for_continuous_2}
        \mathcal{E}_a(u) = \lim_{n\to \infty} \left(\tfrac{5}{3}\right)^{n} \sum_{\omega \in \Sigma_n} \sum_{1 \leq i < j \leq 3} \tfrac{a(\boldsymbol{p}_{\omega i}) + a(\boldsymbol{p}_{\omega j})}{2}\big|u(\boldsymbol{p}_{\omega i})-u(\boldsymbol{p}_{\omega j})\big|^2.
    \end{equation}
    Moreover $\mathcal{E}_a$, defined above, is of class $C^1$ and $(u,v) \longmapsto \langle\mathcal{E}_a'(u),v\rangle$ is an inner product on $H^1_0(V)$, equivalent with the standard one.
\end{lemma}
\begin{proof}
    Let us denote 
    \begin{equation*}
        a_\omega = \min_{\boldsymbol{x} \in F_\omega V} a(\boldsymbol{x}), \quad a_- = \min_{\boldsymbol{x}\in V} a(\boldsymbol{x})\quad \text{and}\quad a_+ = \max_{\boldsymbol{x}\in V} a(\boldsymbol{x}).
    \end{equation*}
    Then we immediately get
    \begin{equation}
        \label{equ:direct_inequalities}
        {a_-}\mathcal{E}(u) \leq \mathcal{E}_a(u) \leq {a_+}\mathcal{E}(u)\quad \text{for every }u\in H^1_0(V).
    \end{equation}
    Moreover, for any word $\omega \in \Sigma_\infty$ and every $k \in \{1,2,3\}$, we clearly have $a_{\omega} \leq a_{\omega k}$ since $\boldsymbol{F}_{\omega k}V \subset \boldsymbol{F}_\omega V$. Using condition \eqref{equ:condition_E^(n)_is_increasing}, for any $n\in \mathbb{N}$, we have
    \begin{equation}
        \label{equ:non-decreasing_E_a}
        \begin{split}
            \left(\tfrac{5}{3}\right)^{n} \sum_{\omega \in \Sigma_n} \sum_{1 \leq i < j \leq 3} a_\omega \big|u(\boldsymbol{p}_{\omega i})-u(\boldsymbol{p}_{\omega j})\big|^2 & \leq \left(\tfrac{5}{3}\right)^{n} \sum_{\omega \in \Sigma_n} \tfrac{5}{3} \sum_{k = 1}^3 \sum_{1 \leq i < j \leq 3} a_\omega \big|u(\boldsymbol{p}_{\omega k i})-u(\boldsymbol{p}_{\omega k j})\big|^2\\
            & \leq \left(\tfrac{5}{3}\right)^{n + 1} \sum_{\omega \in \Sigma_n} \sum_{k = 1}^3 \sum_{1 \leq i < j \leq 3} a_{\omega k} \big|u(\boldsymbol{p}_{\omega k i})-u(\boldsymbol{p}_{\omega k j})\big|^2\\
            & = \left(\tfrac{5}{3}\right)^{n + 1} \sum_{\sigma \in \Sigma_{n+1}} \sum_{1 \leq i < j \leq 3} a_{\sigma} \big|u(\boldsymbol{p}_{\sigma i})-u(\boldsymbol{p}_{\sigma j})\big|^2
        \end{split}
    \end{equation}
    The last equality is achieved by taking $\sigma = \omega k$. By the following inequality
    \begin{equation*}
        \left|a_\omega - a(\boldsymbol{p}_{\omega i})\right| \leq \sup\left\{|a(\boldsymbol{x}) - a(\boldsymbol{y})| : \boldsymbol{x}, \boldsymbol{y}\in V,\ |\boldsymbol{x} - \boldsymbol{y}| \leq \sqrt{3} |\boldsymbol{p}_1 - \boldsymbol{p}_2|2^{-n - 1}\right\}
    \end{equation*}
    and by uniform continuity of $a$, we get 
    \begin{equation*}
        \sup_{\substack{\omega \in \Sigma_n \\ 1\leq i < j \leq 3}}\left|\tfrac{a(\boldsymbol{p}_{\omega i}) + a(\boldsymbol{p}_{\omega j})}{2} - a_\omega\right| \to 0 \quad \text{as }n\to \infty.
    \end{equation*}
    Hence we also get that
    \begin{equation*}
        \left|\left(\tfrac{5}{3}\right)^{n} \sum_{\omega \in \Sigma_n} \sum_{1 \leq i < j \leq 3} \left(\tfrac{a(\boldsymbol{p}_{\omega i}) + a(\boldsymbol{p}_{\omega j})}{2} - a_\omega\right)\big|u(\boldsymbol{p}_{\omega i})-u(\boldsymbol{p}_{\omega j})\big|^2\right|
    \end{equation*}
    is bounded above by a term
    \begin{equation*}
        \sup_{\substack{\omega \in \Sigma_n \\ 1\leq i < j \leq 3}}\left|\tfrac{a(\boldsymbol{p}_{\omega i}) + a(\boldsymbol{p}_{\omega j})}{2} - a_\omega\right| \mathcal{E}(u)
    \end{equation*}
    that converges to $0$ as $n\to \infty$ and consequently we can decompose 
    \begin{equation}
        \label{equ:n-th_part_of_E_a}
        \left(\tfrac{5}{3}\right)^{n} \sum_{\omega \in \Sigma_n} \sum_{1 \leq i < j \leq 3} \tfrac{a(\boldsymbol{p}_{\omega i}) + a(\boldsymbol{p}_{\omega j})}{2}\big|u(\boldsymbol{p}_{\omega i})-u(\boldsymbol{p}_{\omega j})\big|^2 
    \end{equation}
    as a sum of two terms:
    \begin{equation*}
        \begin{split}
            \left(\tfrac{5}{3}\right)^{n} \sum_{\omega \in \Sigma_n} \sum_{1 \leq i < j \leq 3} & a_\omega\big|u(\boldsymbol{p}_{\omega i})-u(\boldsymbol{p}_{\omega j})\big|^2\\
            & + \left(\tfrac{5}{3}\right)^{n} \sum_{\omega \in \Sigma_n} \sum_{1 \leq i < j \leq 3} \left(\tfrac{a(\boldsymbol{p}_{\omega i}) + a(\boldsymbol{p}_{\omega j})}{2} - a_\omega\right)\big|u(\boldsymbol{p}_{\omega i})-u(\boldsymbol{p}_{\omega j})\big|^2 
        \end{split}
    \end{equation*}
    where one is non-decreasing by \eqref{equ:non-decreasing_E_a} and clearly bounded by $a_+\mathcal{E}(u)$ The second one converges to $0$. This shows that the desired limit exists. Hence one can pass to the limit with the parallelogram law described for \eqref{equ:n-th_part_of_E_a} to get 
    \begin{equation*}
        \mathcal{E}_a(u - v) + \mathcal{E}_a(u + v) = 2\mathcal{E}_a(u) + 2 \mathcal{E}_a(v).
    \end{equation*}
    Now it is clear that $\mathcal{E}_a$ is a square of a norm introducing an inner product on $H^1_0(V)$, that is equivalent with the standard one by \eqref{equ:direct_inequalities}. Consequently all desired properties follows.
\end{proof}
To allow discontinuous weight \(a\), the construction of the Laplacian on \(V\) forces us to define \(a\) on the set \(V_\infty\). Accordingly, we impose the following assumption
\begin{equation}
    \label{ass:a}
    \tag{$a$}
    \boxed{
    \begin{tabular}{l}
        \text{function $a\colon V_\infty \longrightarrow (0, \infty)$ satisfy} \\
        \qquad $0 < {a_-} = \inf\limits_{\boldsymbol{x}\in V_\infty} a(\boldsymbol{x}) \leq \sup\limits_{\boldsymbol{x}\in V_\infty} a(\boldsymbol{x}) = {a_+} < \infty$
    \end{tabular}
    }
\end{equation}
and define $\mathcal{E}_a\colon H^1_0(V)\longrightarrow \mathbb{R}$, as a natural generalization of \eqref{equ:potential_of_Delta_a}, by the formula
\begin{equation}
    \label{equ:functional_E_a}
    \mathcal{E}_a(u) = \limsup_{n\to \infty}\left(\tfrac{5}{3}\right)^{n} \sum_{\omega \in \Sigma_n} \sum_{1 \leq i < j \leq 3} \tfrac{a(\boldsymbol{p}_{\omega i}) + a(\boldsymbol{p}_{\omega j})}{2}\big|u(\boldsymbol{p}_{\omega i})-u(\boldsymbol{p}_{\omega j})\big|^2
\end{equation}
for every $u\in H^1_0(V)$. Note that, for every continuous function $a\colon V \longrightarrow (0,\infty)$, the restriction $a|_{V_\infty}$ satisfies \eqref{ass:a}. Therefore, the functional $\mathcal{E}_a$ given by \eqref{equ:functional_E_a} generalizes the weak Laplacian described in Lemma~\ref{lem:E_a_for_continuous_a}. In general, it is not known whether $\mathcal{E}_a$ is of class $C^1$. However, the estimates in \eqref{equ:direct_inequalities} still hold. Some further properties of $\mathcal{E}_a$ are described in Proposition \ref{prop:E_a_satisfies_psi}.
\begin{proposition}
    \label{prop:E_a_satisfies_psi}
    Assume that \eqref{ass:a} holds. Then the functional $\psi\colon H_0^1(V)\longrightarrow\mathbb{R}$ defined by
    \begin{equation}
    \label{equ:psi}
    \psi(u) = \tfrac{1}{2}\mathcal{E}_a(u) \quad \text{for every }u\in H^1_0(V)
    \end{equation}
    with $\mathcal{E}_a$ given by \eqref{equ:functional_E_a}, satisfies \eqref{ass:psi} and \eqref{ass:psi_rho} with $\rho(x) = {a_-} x$.
\end{proposition}
\begin{proof}
    We apply \cite[Proposition 1.1.2]{hiriarturruty1993} to prove condition \eqref{ass:psi_rho}, which is equivalent to convexity of ${\frac{1}{2}\mathcal{E}_a(\cdot) -\frac{{a_-}}{2}\|{\cdot}\|_{H_0^1(V)}^2}$. Note that $\psi - \frac{{a_-}}{2}\mathcal{E}$ is convex since
    \begin{equation*}
        \psi(u) -\tfrac{{a_-}}{2}\mathcal{E}(u) = \tfrac12\limsup_{n\to \infty}\left(\tfrac{5}{3}\right)^{n} \sum_{\omega \in \Sigma_n} \sum_{1 \leq i < j \leq 3} \tfrac{\left(a(\boldsymbol{p}_{\omega i}) + a(\boldsymbol{p}_{\omega j})- 2 {a_-}\right)}{2}\big|u(\boldsymbol{p}_{\omega i})-u(\boldsymbol{p}_{\omega j})\big|^2.
    \end{equation*}
    Indeed, for every fixed $\boldsymbol{x},\boldsymbol{y}\in V$, the mapping $u \longmapsto \frac{a(\boldsymbol{x}) + a(\boldsymbol{y}) - 2{a_-}}{2}\big|u(\boldsymbol{x})-u(\boldsymbol{y})\big|^2$ is convex. Hence each 
    \begin{equation*}
        u\longmapsto \left(\tfrac{5}{3}\right)^{n} \sum_{\omega \in \Sigma_n} \sum_{1 \leq i < j \leq 3} \tfrac{a(\boldsymbol{p}_{\omega i}) + a(\boldsymbol{p}_{\omega j}) - 2{a_-}}{2}\big|u(\boldsymbol{p}_{\omega i})-u(\boldsymbol{p}_{\omega j})\big|^2
    \end{equation*}
    defines a convex functional. Consequently $\psi - \mathcal{E}_a$ is convex as an upper limit of convex functionals. Moreover $\partial \left(\psi - \frac{{a_-}}{2}\mathcal{E}\right)(u) = \partial \psi(u) - \{ {a_-}\mathcal{E}'(u)\}$. Using that and the fact that a subdifferential of a convex functional is monotone we get
    \begin{equation*}
        \begin{split}
            \langle \eta - \xi, u - v\rangle \geq & {a_-} \langle \mathcal{E}'(u) - \mathcal{E}'(v), u - v\rangle = {a_-}\|u - v\|_{H^1_0}^2\\
            & \text{for all }u,v\in H^1_0(V)\text{ and every }\eta \in \partial\psi(u), \xi \in \partial \psi(v),
        \end{split}
    \end{equation*}
    which gives \eqref{ass:psi_rho} by \cite[Lemma 3.1]{Pietrusiak}.
\end{proof}

In general, the functional $\mathcal{E}_a$ need not be G{\^a}teaux differentiable. However, its convexity allows one to consider its subdifferential, in the sense of convex analysis, at every point. Consequently, as in the case of the weak Laplacian, for continuous $a$ the operator ${-\Delta_a \colon H^1_0(V) \longrightarrow \left(H^1_0(V)\right)^*}$ is single-valued. When $a$ is merely measurable, this no longer holds; in that case one obtains a maximally monotone operator ${-\Delta_a \colon H^1_0(V) \longrightarrow 2^{\left(H^1_0(V)\right)^*}}$. Recall that, for a fixed function $f \in L^2(V)$, we will understand the inclusion
\begin{equation}
    \label{equ:inclusion_fixed_f}
    \begin{cases}
        -\Delta_a u(\boldsymbol{x}) \ni f(\boldsymbol{x}) & \text{for every }\boldsymbol{x}\in V\setminus V_0 \\
        u(\boldsymbol{x}) = 0 & \text{for every }\boldsymbol{x}\in V_0
    \end{cases}
\end{equation}
as
\begin{equation*}
    \tfrac{1}{2}\mathcal{E}_a(v) \geq \tfrac{1}{2}\mathcal{E}_a(u) + \int_V f(\boldsymbol{x})\big(v(\boldsymbol{x}) - u(\boldsymbol{x})\big)\, d\mu(\boldsymbol{x}) \quad \text{for every }v\in H^1_0(V).
\end{equation*}
Since $\mathcal{E}_a$ is of class $C^1$ under the assumption that $a$ is continuous, we may consider, instead of the inclusion \eqref{equ:inclusion_fixed_f}, the equation
\begin{equation*}
   \begin{cases}
       -\Delta_au(\boldsymbol{x}) = f(\boldsymbol{x}) & \text{for every }\boldsymbol{x}\in V\setminus V_0\\
       u(\boldsymbol{x}) = 0 & \text{for every }\boldsymbol{x}\in V_0,
   \end{cases}
\end{equation*}
In both cases, solutions are the critical points of the functional
\begin{equation*}
    H^1_0(V) \ni v\longmapsto \tfrac{1}{2}\mathcal{E}_a(v) - \int_V f(\boldsymbol{x}) v(\boldsymbol{x})\, d\mu(\boldsymbol{x}).
\end{equation*}
To consider $f$ depending on $u$, we impose suitable assumptions on $f$.

\subsection{Assumptions on nonlinear term}

Under the following, structural assumption on $f$:
\begin{equation}
    \label{ass:f_Caratheodory}
    \tag{$f_{\mathrm{C}}$}
    \boxed{
    \begin{tabular}{l}
        \text{$f\colon V\times \mathbb{R}\longrightarrow \mathbb{R}$ is \emph{an $L^1$-Carath\'eodory function}, that is:} \\
        \quad \text{$\bullet$ $f(\boldsymbol{\cdot}, u)$ is measurable for every $u\in \mathbb{R}$,} \\
        \quad \text{$\bullet$ $f(\boldsymbol{x}, \cdot)$ is continuous for $\mu$-a.e. $\boldsymbol{x}\in V$,}\\
        \quad \text{$\bullet$ for every $M > 0$ there exists $f_M\in L^1(V; \mathbb{R}_+)$ such that}\\
        \quad \qquad $\left|f(\boldsymbol{x}, u)\right| \leq f_M(\boldsymbol{x})$\text{ for $\mu$-a.e. }$\boldsymbol{x}\in V$\text{ and every }$u\in [-M,M]$.
    \end{tabular}
    }
\end{equation}
we define an antiderivative of $f(\boldsymbol{x}, \cdot)$ by the formula
\begin{equation}
    \label{equ:F-primitive_of_f}
    F(\boldsymbol{x},u) =  \int_{0}^{u} f(\boldsymbol{x},s)\, ds\quad \text{for all }u\in \mathbb{R}\text{ and $\mu$-a.e. }\boldsymbol{x}\in V.
\end{equation}
The functional $\Phi$ is defined in the following lemma.
\begin{lemma}
    \label{lem:F_satisfies_Phi}
    Assume that \eqref{ass:f_Caratheodory} holds. Then the functional $\Phi\colon H_0^1(V)\longrightarrow \mathbb{R}$ defined by
    \begin{equation}
        \label{equ:Phi}
        \Phi(u)= 
        -\int_V F\big(\boldsymbol{x},u(\boldsymbol{x})\big)\, d\mu(\boldsymbol{x})\quad \text{for every }u\in H^1_0(V)
    \end{equation}
    satisfies \eqref{ass:Phi}. Moreover
    \begin{equation}
        \label{equ:Phi_derivative}
        \left\langle \Phi'(u), v\right\rangle = -\int_V f\big(\boldsymbol{x}, u(\boldsymbol{x})\big)v(\boldsymbol{x})\, d\mu(\boldsymbol{x})\quad \text{for every }u,v\in H^1_0(V).
    \end{equation}
\end{lemma}

\begin{proof}
    Let $u,v\in H_0^1(V)$. By \eqref{ass:f_Caratheodory} we have
    \begin{equation*}
        \lim\limits_{t\to0}\frac{F\big(\boldsymbol{x},u(\boldsymbol{x})+tv(\boldsymbol{x})\big)-F\big(\boldsymbol{x},u(\boldsymbol{x})\big)}{t}=f\big(\boldsymbol{x},u(\boldsymbol{x})\big)v(\boldsymbol{x})\quad \text{for $\mu$-a.e. }\boldsymbol{x}\in V.
    \end{equation*}
    The embedding $H_0^1(V) \hookrightarrow C(V)$ gives $\|u\|_\infty , \|v\|_\infty \leq M$ for some $M>0$ and hence the Lagrange mean value theorem yields
    \begin{equation*}
        \left|\frac{F\big(\boldsymbol{x},u(\boldsymbol{x})+tv(\boldsymbol{x})\big)-F\big(\boldsymbol{x},u(\boldsymbol{x})\big)}{t}\right|\le\max_{s\in[-2M,2M]}|f(\boldsymbol{x},s)|\cdot M \leq  M f_{2M}(\boldsymbol{x})
    \end{equation*}
    for $\mu$-a.e. $\boldsymbol{x}\in V$ provided $-1 \leq t \leq 1$. Using already mentioned imbedding $H^1_0(V) \hookrightarrow C(V)$ together with \cite[Theorem 2.1]{IdczakRogowski} we instantly get continuity of $\Phi'$ since for every $u, v \in H^1_0(V)$ one has
    \begin{equation*}
        \begin{split}
            \left|\langle \Phi'(u), v\rangle\right| & = \int_V \left|f\big(\boldsymbol{x}, u(\boldsymbol{x})\big) v(\boldsymbol{x})\right|\, d\mu(\boldsymbol{x})\\ 
            & \leq \|v\|_\infty \int_V \left|f\big(\boldsymbol{x}, u(\boldsymbol{x})\big) \right|\, d\mu(\boldsymbol{x}) \leq 9 \|v\|_{H^1_0}\left\|f\big(\boldsymbol{\cdot}, u(\boldsymbol{\cdot})\big)\right\|_{L^1}
        \end{split}
    \end{equation*}
    by Lemma \ref{lemma:embedding}.
\end{proof}

We consider also the growth condition of the following form
\begin{equation}
    \label{ass:f_sublinear}
    \tag{$f_{\infty}$}
    \boxed{
    \begin{tabular}{l}
        \text{there exists $\alpha < {a_-}\lambda_1$ and $\beta \in L^1(V; \mathbb{R}_+)$ such that}\\
        \qquad $F(\boldsymbol{x}, u) \leq \frac{\alpha}{2} |u|^2 + \beta(\boldsymbol{x})$\text{ for all $u\in \mathbb{R}$ and $\mu$-a.e. $\boldsymbol{x}\in V$}
    \end{tabular}
    }
\end{equation}
that provides the lower bound for $\Phi$ of the following form.

\begin{lemma}\label{lem: coercivity}
    Assume that \eqref{ass:f_Caratheodory} and \eqref{ass:f_sublinear} hold. Then functional $\Phi$ defined by \eqref{equ:Phi} satisfies
    \begin{equation*}
        \Phi(u) \geq -\tfrac{\alpha}{2\lambda_1}\|u\|_{H^1_0}^2 - \|\beta\|_{L^1}\quad \text{for every }u\in H^1_0(V).
    \end{equation*}
\end{lemma}
\begin{proof}
    Using the Poincar{\'e} inequality \eqref{equ:Poincare_inequality} we get
    \begin{equation*}
        \begin{split}
            \Phi(u) & =-\int_VF(\boldsymbol{x},u(\boldsymbol{x}))\, d\mu\geq -\int_V \left( \tfrac{\alpha}{2}|u(\boldsymbol{x})|^2+\beta(\boldsymbol{x})\right)\, d\mu\\
            &= -\tfrac{\alpha}{2}\|u\|_{L^2}^2-\|\beta\|_{L^1}\ge -\tfrac{\alpha}{2\lambda_1}\|u\|_{H_0^1}^2-\|\beta\|_{L^1}\quad \text{for every }u\in H^1_0(V).\quad\qedhere
        \end{split}
    \end{equation*}
\end{proof}

\begin{remark}
    Assumption \eqref{ass:f_sublinear} can be easily reformulated in terms of the function $f$. It is sufficient to assume that \emph{there exist a constant $\alpha < a_- \lambda_1$ and a function $\beta \in L^1(V;\mathbb{R}_+)$ such that}
    \begin{equation*}
	   |f(\boldsymbol{x},u)| \le \alpha |u| + \beta(\boldsymbol{x}) \quad \text{for all } u \in \mathbb{R} \text{ and for $\mu$-a.e. } \boldsymbol{x}\in V .
    \end{equation*}
    Moreover, the above assumption is in particular satisfied whenever \emph{there exist constants $\alpha>0$, $r\in(0,1)$ and a function $\beta\in L^1(V;\mathbb{R}_+)$ such that}
    \begin{equation*}
	   f(\boldsymbol{x},u) \le \alpha |u|^{r} + \beta(\boldsymbol{x}) \quad \text{for all } u \in \mathbb{R} \text{ and for $\mu$-a.e. } \boldsymbol{x}\in V .
    \end{equation*}
\end{remark}

We consider the following assumption
\begin{equation}
    \label{ass:f_0}
    \tag{$f_0$}
    \boxed{
    \lim_{u \to 0} \frac{F(\boldsymbol{x}, u)}{|u|^2} = \infty \quad \text{uniformly with respect to $\mu$-a.e. }\boldsymbol{x}\in V
    }
\end{equation}
that allows to describe behaviour of $\Phi$ near $0$ in the following way.
\begin{lemma}
    \label{lem:small_near_zero}
    Assume that \eqref{ass:f_Caratheodory}, \eqref{ass:f_0} hold and define $\Phi\colon H^1_0(V) \longrightarrow \mathbb{R}$ using \eqref{equ:Phi}. Then for all $u\in H^1_0(V)$ and $c > 0$ there exists $t^* > 0$ such that
    \begin{equation*}
        \Phi(tu) < -c t^2\quad \text{for every }t\in (0,t^*).
    \end{equation*}
\end{lemma}
\begin{proof}
    Let $\gamma = c\|u\|_{L^2}^{-2}$. By assumption \eqref{ass:f_0}, there exists $r > 0$ such that
    \begin{equation*}
        F(\boldsymbol{x}, u) \geq \gamma |u|^2\quad \text{for all }u\in [-r,r]\text{ and $\mu$-a.e. }\boldsymbol{x}\in V.
    \end{equation*}
    Since $H^1_0(V) \hookrightarrow C(V)$, we can take $t^* = \min \left(r\|u\|_{\infty}^{-1},1\right)$. Then for any $t\in (0,t^*)$ we get $\|t u\|_\infty < r$ and hence
    \begin{equation*}
        \Phi(tu) = -\int_V F\big(\boldsymbol{x}, tu(\boldsymbol{x})\big)\, d\mu(\boldsymbol{x}) \leq -\int_V \gamma |tu(\boldsymbol{x})|^2\, d\mu(\boldsymbol{x}) = -\gamma t^2 \|u\|_{L^2}^2 = -c t^2
    \end{equation*}
    for every $t\in (0, t^*)$.
\end{proof}

\begin{example}
    Taking $f\colon V \times \mathbb{R} \longrightarrow \mathbb{R}$ given by $f(\boldsymbol{x}, u) = |u|^{-\frac{1}{2}}u$ we get a standard nonlinearity satisfying \eqref{ass:f_0}.
\end{example}

Let us recall that the Ambrosetti-Rabinovitz condition:
\begin{equation}
    \label{ass:f_Ambrosetti-Rabinowitz}
    \tag{$f_{\mathrm{AR}}$}
    \boxed{
    \begin{tabular}{l}
         \text{there exist $\theta > 2$ and $M_0 > 0$ such that}\\
         \qquad $0 < \theta F(\boldsymbol{x}, u) \leq uf(\boldsymbol{x}, u)$\\
         \text{for all $u\in (-\infty, -M_0] \cup [M_0, \infty)$ and $\mu$-a.e. $\boldsymbol{x}\in V$}
    \end{tabular}
    }
\end{equation}
leads to the following lower bound for $F$.
\begin{lemma}\label{lem: AR_consequence}
    Assume that \eqref{ass:f_Caratheodory}, \eqref{ass:f_Ambrosetti-Rabinowitz} hold. Then there exist $A>0$ and $B\in L^1(V),$ such that 
    \begin{equation*}
        F(\boldsymbol{x},u)\geq A|u|^\theta-B(\boldsymbol{x}) \quad \text{ for all $u\in\mathbb{R}$ and $\mu$-a.e. $\boldsymbol{x}\in V$}.
    \end{equation*}
\end{lemma}
\begin{proof}
    We use standard arguments analogous to those in \cite[Remark~7.5]{Jabri}, adapted to the Sierpiński triangle $V$. Let us start by setting $M = \max\{1, M_0\}$ and fixing $\boldsymbol{x} \in V$ for which \eqref{ass:f_Ambrosetti-Rabinowitz} holds. Then we obtain the inequality 
    \begin{equation}
        \label{equ:lem AR_consequence:proof}
        \frac{\theta}{v} \leq \frac{f(\boldsymbol{x},v)}{F(\boldsymbol{x},v)}, \quad \text{if } |v| \geq M.
    \end{equation}
    Integrating it over the interval $[M, u]$, with $u > M$, with respect to $v$, we get $\theta \ln \frac{u}{M_0} \leq \ln \frac{F(\boldsymbol{x},u)}{F(\boldsymbol{x},M_0)}$. Hence, 
    \begin{equation}
        \label{equ:lem AR_consequence:proof-2}
       A |u|^{\theta} \leq F(\boldsymbol{x},u)
    \end{equation}
    for $u \geq M$, where $A = \frac{F(\boldsymbol{x},M_0)}{M_0^{\theta}}$. Integrating \eqref{equ:lem AR_consequence:proof} over $[v, -M]$, with $v < -M$, shows that inequality \eqref{equ:lem AR_consequence:proof-2} holds also for $v \leq -M$. Finally, to obtain the assertion, it is enough to take $B = f_M$, where $f_M$ is the bound from assumption \eqref{ass:f_Caratheodory}.
\end{proof}

Together with \eqref{ass:f_Ambrosetti-Rabinowitz}, it is typical to assume the following condition on $F$
\begin{equation}
    \label{equ:limsup_F=0}
    \limsup_{u \to 0}\frac{F(\boldsymbol{x},u)}{|u|^2} = 0 \quad \text{uniformly for $\mu$-a.e. } \boldsymbol{x} \in V,
\end{equation}

\begin{example}
    \label{exam:2}
    Let $f\colon V \times \mathbb{R} \longrightarrow \mathbb{R}$ be given by the formula $f(\boldsymbol{x}, u) = |u|u$. Then $F(\boldsymbol{x}, u) = \frac{1}{3}|u|^3$ and hence \eqref{equ:limsup_F=0} holds. Assumption \eqref{ass:f_Ambrosetti-Rabinowitz} holds as well for $\theta = 3$ and any $M_0 > 0$.
\end{example}

Condition \eqref{ass:f_Ambrosetti-Rabinowitz} is clearly incompatible with
\eqref{ass:f_0}. Motivated by \cite{Bereanu2012}, we therefore propose an
alternative assumption. Observe that whenever \eqref{equ:limsup_F=0} holds,
there exists $R>0$ sufficiently small such that
\begin{equation*}
    F(\boldsymbol{x},u)
    \le \tfrac{a_-|u|^2}{163}
    \le \tfrac{a_-R^2}{163}
    < \tfrac{a_-R^2}{162}
\end{equation*}
for all $u\in [-R,R]$ and for $\mu$-a.e. $\boldsymbol{x} \in V$. This motivates the following assumption:
\begin{equation}
    \label{ass:f_R}
    \tag{$f_R$}
    \boxed{
    \begin{tabular}{l}
        \text{there exist $R>0$ and $\alpha<\frac{a_-R^2}{162}$ such that}\\
        \qquad $F(\boldsymbol{x},u)\le \alpha$\\
        \text{for all $u\in[-R,R]$ and for $\mu$-a.e. $\boldsymbol{x}\in V$},
    \end{tabular}
    }
\end{equation}
which is therefore satisfied. Observe that the seemingly unusual constant $162$ is simply the product of the standard variational factor $\frac{1}{2}$ and the square of the inverse of the constant $9$ appearing in the embedding $H^1(V) \hookrightarrow C(V)$. Consequently, the function $f$ described in Example~\ref{exam:2} satisfies both \eqref{equ:limsup_F=0} and \eqref{ass:f_Ambrosetti-Rabinowitz}. Below we provide an explicit example of a nonlinearity satisfying simultaneously \eqref{ass:f_0}, \eqref{ass:f_Ambrosetti-Rabinowitz}, and \eqref{ass:f_R}.

\begin{example}
    Consider, for simplicity, $a \equiv 1$, and let $f\colon V \times \mathbb{R}\longrightarrow \mathbb{R}$ be defined by
    \begin{equation*}
        f(\boldsymbol{x},u)=
        \begin{cases}
            \frac{1}{11} |u|^{-\frac{1}{2}}u &\text{ if } |u|\leq 100,\\
            \frac{1}{11\,000} |u|u &\text{ if } |u|>100.
        \end{cases}
    \end{equation*}
    Assumption \eqref{ass:f_0} holds, since $F(\boldsymbol{x}, u) = \frac{3}{22}|u|^\frac{3}{2}$ for $u\in [-100,100]$ and $\boldsymbol{x}\in V$. Hence
    \begin{equation*}
        \lim_{u\to 0}\frac{F(\boldsymbol{x}, u)}{|u|^2} = \lim_{u\to 0}\frac{3}{22\sqrt{|u|}} = \infty.
    \end{equation*}
    If $|u| > 100$ and $\boldsymbol{x}\in V$, then $F(\boldsymbol{x}, u) = \frac{|u|^3}{33\, 000} + \frac{1000}{33}$. To verify that \eqref{ass:f_Ambrosetti-Rabinowitz} holds, we take $\theta = \frac{23}{8} > 2$. The function $F$ is positive for $u\neq 0$. A straightforward computation shows that $\theta F(\boldsymbol{x}, u) \leq u f(\boldsymbol{x}, u)$ is equivalent to
    \begin{equation*}
        \frac{1000}{33} \leq \frac{|u|^3}{24\cdot 11\, 000}
    \end{equation*}
    for all $u\in (-\infty, -100)\cup (100, \infty)$. Therefore, it suffices to take $M_0 = 200$ in \eqref{ass:f_Ambrosetti-Rabinowitz}. Next, since we assumed that $a \equiv 1$, we have $a_- = a_+ = 1$. Moreover, taking $R = 100$ we obtain
    \begin{equation*}
        F(\boldsymbol{x}, u) \leq F(\boldsymbol{x}, R) = \frac{2000}{33} < 61 < \frac{10\, 000}{162} = \frac{a_- R^2}{162}
    \end{equation*}
    for every $\boldsymbol{x}\in V$ and $u\in [-R,R]$. Consequently, \eqref{ass:f_R} holds.
\end{example}

Assumption \eqref{ass:f_R} allows to better control a geometry of $\Phi$ as it is described in the following
\begin{lemma}
    \label{lem:mountains}
    Assume that \eqref{ass:f_Caratheodory}, \eqref{ass:f_R} hold and define $\Phi \colon H^1_0(V)\longrightarrow \mathbb{R}$ using \eqref{equ:Phi}. Then for every $u\in H^1_0(V)$ with $\|u\|_{H^1_0} = \frac{R}{9}$ we have $\Phi(u) \geq -\alpha$.
\end{lemma}
\begin{proof}
    Take $u$ satisfying assumptions of the lemma. Then $\|u\|_\infty \leq 9\|u\|_{H^1_0} = R$ and hence
    \begin{equation*}
        \Phi(u)  = -\int_V F\big(\boldsymbol{x}, u(\boldsymbol{x})\big)\, d\mu(\boldsymbol{x}) \geq -\int_V \alpha\, d\mu(\boldsymbol{x}) = -\alpha. \qedhere
    \end{equation*}
\end{proof}

\section{Solvability of main problem}
\label{sec: Main result}
Analogously to the case of the nonlinear Laplace equation with Dirichlet boundary condition described in \cite{strichartzwong2004}, we call $u\in H^1_0(V)$ a solution to the equation 
\begin{equation}
    \label{equ:main_problem}
    \begin{cases}
        -\Delta_a u(\boldsymbol{x}) \ni f\big(\boldsymbol{x}, u(\boldsymbol{x})\big) & \text{for }\boldsymbol{x}\in V\setminus V_0,\\
        u(\boldsymbol{x}) = 0 & \text{for }\boldsymbol{x}\in V_0,
    \end{cases}
\end{equation}
if the function is a critical point of the functional $I\colon H^1_0(V) \longrightarrow \mathbb{R}$ defined by
\begin{equation}
    \label{equ:I}
    I(u) = \tfrac12\mathcal{E}_a(u)-\int_V F\big(\boldsymbol{x},u(\boldsymbol{x})\big)\, d\mu(\boldsymbol{x})\quad \text{for every }u\in H^1_0(V)
\end{equation}
with $\mathcal{E}_a$ defined by \eqref{equ:functional_E_a} and $F$ given by \eqref{equ:F-primitive_of_f}. Equivalently $u$ solves \eqref{equ:main_problem}, if the following hemivariational inequality holds
\begin{equation*}
   \int_V f(\boldsymbol{x},u(\boldsymbol{x}))\big(v(\boldsymbol{x})-u(\boldsymbol{x})\big)\, d\mu(\boldsymbol{x}) +\tfrac12\mathcal{E}_a(v) \geq \tfrac12\mathcal{E}_a(u)\quad \text{for all }v\in H_0^1(V).
\end{equation*}
Consequently $I$ has the form \eqref{equ:Szulkin-type_functional} if we let $\psi$ and $\Phi$ be given by \eqref{equ:psi} and \eqref{equ:Phi}, respectively. We begin with a basic result guaranteeing the solvability of
equation~\eqref{equ:main_problem}.

\begin{theorem}
    \label{thm:existence}
    Assume that \eqref{ass:a}, \eqref{ass:f_Caratheodory} and \eqref{ass:f_sublinear} hold. Then problem \eqref{equ:main_problem} has at least one solution $u^*$ satisfying
    \begin{equation}
        \label{equ:minimization_I}
        I(u^*) = \min_{u\in H^1_0(V)}I(u)
    \end{equation}
    with $I$ given by \eqref{equ:I}.
\end{theorem}
\begin{proof}
    We define $I$ using formula \eqref{equ:I}. By applying Proposition~\ref{prop:E_a_satisfies_psi} and Lemma~\ref{lem:F_satisfies_Phi}, we see that we are in a position to use Proposition~\ref{prop:I_coercive=>(PS)} (assumption \eqref{ass:E} holds trivially, since $H^1_0(V)$ is a Hilbert space). Therefore, it suffices to show that 
    \begin{equation*}
        \lim_{\|u\|_{H^1_0} \to \infty} I(u) = \infty.
    \end{equation*}
    It follows immediately from estimates \eqref{equ:direct_inequalities} together with Lemma~\ref{lem: coercivity}. Indeed,
    \begin{equation*}
        I(u) \geq \tfrac{a_-}{2} \mathcal{E}(u) - \tfrac{\alpha}{2 \lambda_1} \|u\|_{H^1_0}^2 - \|\beta\|_{L^1} = \tfrac{1}{2} \left(a_- - \tfrac{\alpha}{\lambda_1}\right) \|u\|_{H^1_0}^2 - \|\beta\|_{L^1} \quad \text{for every } u \in H^1_0(V).
    \end{equation*}
    The argument of a global minimum is always a critical point and hence, the existence of a solution is obtained.
\end{proof}

Before proceeding further, let us prove the following technical lemma, which, combined with Proposition \ref{prop:Mawhin2}, will allow us to establish the existence of two nontrivial solutions.
 
\begin{lemma}
    \label{lem: P-S condition}
     Assume that \eqref{ass:f_Caratheodory} and \eqref{ass:f_Ambrosetti-Rabinowitz} hold and let $I$ be given by \eqref{equ:I}. Then every Palais-Smale sequence $(u_n)\subset H_0^1(V)$ for the functional $I$ is bounded. 
\end{lemma}
\begin{proof}
    Let $(u_n)\subset H_0^1(V)$ be the Palais-Smale sequence for $I$. It means that there exists a sequence $(\varepsilon_n)$ satisfying $\varepsilon_n\to 0$ and such that 
    \begin{equation*}
        \psi(v)-\psi(u_n)+\langle\Phi'(u_n),v-u_n\rangle\geq -\varepsilon_n\|v-u_n\|_{H_0^1(V)}\quad \text{for all }v\in H_0^1(V),
    \end{equation*}
    where $\psi$ and $\Phi$ are defined by \eqref{equ:psi} and \eqref{equ:Phi}, respectively. In particular, taking $v = (1 + t)u_n$, dividing both sides by $t$ and passing to the limit $t \downarrow 0$ we get 
    \begin{equation*}
        \mathcal{E}_a(u_n) - \int_V f\big(\boldsymbol{x}, u_n(\boldsymbol{x})\big) u_n(\boldsymbol{x})\, d\mu(\boldsymbol{x}) \geq -\varepsilon_n\|u_n\|_{H_0^1}\quad \text{for every }n\in \mathbb{N}
    \end{equation*}
    and hence 
    \begin{equation}
        \label{equ:proof_AR_property-1}
        -\tfrac{1}{\theta} \mathcal{E}_a(u_n) + \tfrac{1}{\theta} \int_V f\big(\boldsymbol{x}, u_n(\boldsymbol{x})\big) u_n(\boldsymbol{x})\, d\mu(\boldsymbol{x}) \leq \tfrac{\varepsilon_n}{\theta}\|u_n\|_{H_0^1}\quad \text{for every }n\in \mathbb{N}
    \end{equation}
    Now take $M_0$ from assumption \eqref{ass:f_Ambrosetti-Rabinowitz}. Then for every $\boldsymbol{x}\in V$ either:
    \begin{itemize}
        \item $|u_n(\boldsymbol{x})| \geq M_0$ and then 
        \begin{equation*}
            \tfrac{1}{\theta} f\big(\boldsymbol{x}, u_n(\boldsymbol{x})\big) u_n(\boldsymbol{x}) - F\big(\boldsymbol{x}, u_n(\boldsymbol{x})\big) \geq 0,
        \end{equation*}
        \item or $|u_n(\boldsymbol{x})| < M_0$ and consequently, using function $f_{M_0}$ from assumption \eqref{ass:f_Caratheodory}, we get
        \begin{equation*}
            \begin{split}
                \left|\tfrac{1}{\theta} f\big(\boldsymbol{x}, u_n(\boldsymbol{x})\big) u_n(\boldsymbol{x}) - F\big(\boldsymbol{x}, u_n(\boldsymbol{x})\big)\right| & \leq \tfrac{1}{\theta} \left| f\big(\boldsymbol{x}, u_n(\boldsymbol{x})\big) u_n(\boldsymbol{x}) \right| + \left| F\big(\boldsymbol{x}, u_n(\boldsymbol{x})\big)\right| \\
                & \leq \tfrac{1}{\theta} f_{M_0}(\boldsymbol{x}) M_0 + \left|\int_0^{u(\boldsymbol{x})} f(\boldsymbol{x}, s)\, ds\right| \\
                & \leq \tfrac{1 + \theta}{\theta} M_0 f_{M_0}(\boldsymbol{x}).
            \end{split}
        \end{equation*}
    \end{itemize}
    Therefore, for every $n\in \mathbb{N}$, we get
    \begin{equation}
        \label{equ:proof_AR_property-2}
        \begin{split}
        \int_V \left(\tfrac{1}{\theta} f\big(\boldsymbol{x}, u_n(\boldsymbol{x})\big) u_n(\boldsymbol{x}) - F\big(\boldsymbol{x}, u_n(\boldsymbol{x})\big)\right)\, d\mu(\boldsymbol{x}) & \geq - \int_V \tfrac{1 + \theta}{\theta} M_0 f_{M_0}(\boldsymbol{x})\, d\mu(\boldsymbol{x})\\
        & = -\tfrac{1 + \theta}{\theta} M_0 \|f_{M_0}\|_{L^1}.
        \end{split}
    \end{equation}
    Now adding $I(u_n)$ to both sides of \eqref{equ:proof_AR_property-1} gives
    \begin{equation*}
        \begin{split}
            I(u_n) + \tfrac{\varepsilon_n}{\theta}\|u_n\|_{H^1_0} & \geq \tfrac{\theta - 2}{2\theta} \mathcal{E}_a(u_n) + \int_V \left(\tfrac{1}{\theta} f\big(\boldsymbol{x}, u_n(\boldsymbol{x})\big) u_n(\boldsymbol{x}) - F\big(\boldsymbol{x}, u_n(\boldsymbol{x})\big)\right)\, d\mu(\boldsymbol{x}) \\ 
            & \geq \tfrac{\theta - 2}{2\theta} \mathcal{E}_a(u_n) - \tfrac{1 + \theta}{\theta} M_0 \|f_{M_0}\|_{L^1} \\
            & \geq a_- \tfrac{\theta - 2}{2\theta} \|u_n\|_{H^1_0}^2 - \tfrac{1 + \theta}{\theta} M_0 \|f_{M_0}\|_{L^1}\quad \text{for every }n\in \mathbb{N}.
        \end{split}
    \end{equation*}
    As a consequence we obtain
    \begin{equation*}
        I(u_n) \geq a_- \tfrac{\theta - 2}{2\theta} \|u_n\|_{H^1_0}^2 - \tfrac{\varepsilon_n}{\theta}\|u_n\|_{H^1_0} - \tfrac{1 + \theta}{\theta} M_0 \|f_{M_0}\|_{L^1}\quad \text{for every }n\in \mathbb{N}.
    \end{equation*}
    Now, supposing that $\|u_{n_k}\|_{H^1_0} \to \infty$ for some increasing sequence $(n_k)$, we get $I(u_{n_k}) \to \infty$ as well, but this clearly contradicts the assumption that $(u_n)$ is the Palais-Smale sequence for $I$. Hence $(u_n)$ is bounded.
\end{proof}

With the above lemma at hand, we turn to the next result concerning the existence of a nontrivial solution to problem~\eqref{equ:main_problem}. The proof is based on abstract results which implicitly rely on the direct method of the calculus of variations and on the Mountain Pass Theorem.

\begin{theorem}
    \label{thm:nonzero_solution}
    Assume that \eqref{ass:a} and \eqref{ass:f_Caratheodory} hold. If additionally one of the following condition holds:
    \begin{itemize}
        \item assumptions \eqref{ass:f_sublinear} and \eqref{ass:f_0} are satisfied or
        \item assumptions \eqref{ass:f_R} and \eqref{ass:f_Ambrosetti-Rabinowitz} hold,
    \end{itemize}
    then problem \eqref{equ:main_problem} has at least one nonzero solution.
\end{theorem}
\begin{proof}
    We split the proof into two cases:
    \begin{itemize}
        \item If assumptions~\eqref{ass:f_sublinear} and~\eqref{ass:f_0} are satisfied, then Theorem~\ref{thm:existence} guarantees the existence of a solution $u^*$ satisfying condition~\eqref{equ:minimization_I}. Note that $I(0)=0$, and let $u\in H^1_0(V)\setminus\{0\}$ be arbitrary. Setting $c=\|u\|_{H^1_0}^2$ in Lemma~\ref{lem:small_near_zero} and taking any $t\in(0,t^*)$, we obtain
        \begin{equation}
            \label{equ:proof:thm:nonzero_solution}
            I(tu)=\psi(tu)+\Phi(tu)<\tfrac{t^2}{2}\|u\|_{H^1_0}^2 - t^2\|u\|_{H^1_0}^2 = -\tfrac{t^2}{2}\|u\|_{H^1_0}^2 < 0.
        \end{equation}
        Consequently, $I(0) > I(u^*)$ and hence $u^* \neq 0$.
        \item If, on the other hand, assumptions~\eqref{ass:f_R} and~\eqref{ass:f_Ambrosetti-Rabinowitz} hold, then the conclusion follows directly from Proposition~\ref{prop:Mawhin2}. Indeed, by combining Lemmas~\ref{lem:every_(PS)_is_bounded=>I_(PS)},~\ref{lem:F_satisfies_Phi},~\ref{lem:mountains}~and~\ref{lem: P-S condition}, and Proposition~\ref{prop:E_a_satisfies_psi}, we get that the functional $I$ satisfies the Palais--Smale condition. It therefore remains to verify assumption~\eqref{equ:Mountain_Pass_geometry_for_I}. To this end, set $r=\frac{R}{9}$. Then, by Lemma~\ref{lem:mountains}, we obtain
        \begin{equation*}
            I(u)=\tfrac{1}{2}\mathcal{E}_a(u)-\alpha \geq \tfrac{a_-}{2}\|u\|_{H^1_0}^2-\alpha = \tfrac{a_-R^2}{162}-\alpha > 0 = I(0) \quad \text{for every } u\in \partial B_r.
        \end{equation*}
        The second condition in~\eqref{equ:Mountain_Pass_geometry_for_I} follows from Lemma~\ref{lem: AR_consequence}. Indeed, let $u\in H^1_0(V)\setminus\{0\}$ and $t>0$. Using the estimate obtained there, we compute
        \begin{equation*}
            I(tu)=\tfrac{1}{2}\mathcal{E}_a(tu)-\int_V F\big(\boldsymbol{x},tu(\boldsymbol{x})\big)\,d\mu(\boldsymbol{x}) \leq \tfrac{t^2}{2}\mathcal{E}_a(u)-A t^\theta \|u\|_{L^\theta}^\theta-\|B\|_{L^1}.
        \end{equation*}
        Letting $t\to\infty$, we obtain $I(tu)\to -\infty$. Hence, by choosing $e=tu$ for $t$ sufficiently large, condition~\eqref{equ:Mountain_Pass_geometry_for_I} is satisfied. Consequently, Proposition~\ref{prop:Mawhin2} applies and yields the desired conclusion. \qedhere
    \end{itemize}
\end{proof}

By combining the proofs of Theorems~\ref{thm:existence} and~\ref{thm:nonzero_solution}, and by applying Proposition~\ref{prop:Mawhin2}, we can prove the following result.

\begin{theorem}
    \label{thm:two_nonzero_solutions}
    Assume that \eqref{ass:a}, \eqref{ass:f_Caratheodory}, \eqref{ass:f_0}, \eqref{ass:f_R} and \eqref{ass:f_Ambrosetti-Rabinowitz} hold. Then problem \eqref{equ:main_problem} has at least two nonzero solutions.
\end{theorem}

\begin{proof}
    In the second part of the proof of Theorem~\ref{thm:nonzero_solution} we have already shown that, under the assumptions of the present theorem, the functional $I$ satisfies the Palais--Smale condition and that \eqref{equ:Mountain_Pass_geometry_for_I} holds. In order to apply Proposition~\ref{prop:Mawhin2}, it therefore remains to verify \eqref{equ:condition_for_third_critical_point}. This is, however, an immediate consequence of Lemma~\ref{lem:small_near_zero}. Indeed, fix an arbitrary nonzero element $u \in H^1_0(V)$. By Lemma~\ref{lem:small_near_zero}, repeating the computation in \eqref{equ:proof:thm:nonzero_solution}, we obtain that there exists $t^* > 0$ such that $I(tu) < 0$ for every $t \in (0,t^*)$. Hence, by Proposition~\ref{prop:Mawhin2}, the problem~\eqref{equ:main_problem} admits at least two nontrivial solutions.  
\end{proof}

\section*{Declarations}

The authors declare that no funds, grants, or other support were received during the preparation of this manuscript.

\section*{Acknowledgements}

This paper has been completed while the second author was a Doctoral Candidate in the Interdisciplinary Doctoral School at the Lodz University of Technology, Poland. The author's contribution to this research was estimated at 50\%.


\begin{thebibliography}{99}

\bibitem{Barnsley}
M.~F. Barnsley,
\emph{Fractals everywhere},
Academic Press, Boston, MA, 1988.

\bibitem{Pietrusiak}
M. Be{\l}dzi{\'n}ski, M. Galewski and F. Pietrusiak,
Minimization principle for hemivariational--variational inequality driven by uniformly monotone operators with application to problems in contact mechanics,
\emph{Nonlinear Analysis: Real World Applications} \textbf{79} (2024), 104134.

\bibitem{Bereanu2012}
C. Bereanu, P. Jebelean and J. Mawhin,
Multiple solutions for Neumann and periodic problems with singular $\varphi$-Laplacian,
\emph{Nonlinear Analysis} \textbf{75} (2012), 731--740.

\bibitem{BonannoBisciRadulescu2012}
G. Bonanno, G. Molica Bisci and V. R\u{a}dulescu,
Variational analysis for a nonlinear elliptic problem on the Sierpi\'nski gasket,
\emph{ESAIM: Control, Optimisation and Calculus of Variations} \textbf{18} (2012), no.~4, 941--953.

\bibitem{Falconer1999}
K.~J. Falconer and J. Hu,
Non-linear elliptic equations on the Sierpi\'nski gasket,
\emph{Journal of Mathematical Analysis and Applications} \textbf{240} (1999), no.~2, 552--573.

\bibitem{Figueiredo}
D.~G. de Figueiredo,
\emph{Lectures on the Ekeland Variational Principle with Applications and Detours},
Lectures on Mathematics and Physics, no.~81,
Springer, Berlin, 1989.

\bibitem{Fukushima1992}
M. Fukushima and T. Shima,
On a spectral analysis for the Sierpi\'nski gasket,
\emph{Potential Analysis} \textbf{1} (1992), 1--35.

\bibitem{Galewski-book}
M. Galewski,
\emph{Basic Monotonicity Methods with Some Applications},
Compact Textbooks in Mathematics,
Birkh\"auser, Cham, 2021.

\bibitem{Galewski2019}
M. Galewski,
On the mountain pass solutions to boundary value problems on the Sierpi\'nski gasket,
\emph{Results in Mathematics} \textbf{74} (2019), no.~4, 167.

\bibitem{Galewski2019b}
M. Galewski,
On the application of monotonicity methods to the boundary value problems on the Sierpi\'nski gasket,
\emph{Numerical Functional Analysis and Optimization} \textbf{40} (2019), no.~11, 1344--1354.

\bibitem{Galewski}
M. Galewski,
On variational nonlinear equations with monotone operators,
\emph{Advances in Nonlinear Analysis} \textbf{10} (2021), 289--300.

\bibitem{GibbonsRajStrichartz}
M. Gibbons, A. Raj and R.~S. Strichartz,
The finite element method on the Sierpi\'nski gasket,
\emph{Constructive Approximation} \textbf{17} (2001), no.~4, 561--588.

\bibitem{hiriarturruty1993}
J.-B. Hiriart-Urruty and C. Lemar{\'e}chal,
\emph{Convex Analysis and Minimization Algorithms I: Fundamentals},
Springer, Berlin, 1993.

\bibitem{Hu2004}
J. Hu,
Multiple solutions for a class of nonlinear elliptic equations on the Sierpi\'nski gasket,
\emph{Science in China Series A: Mathematics} \textbf{47} (2004), no.~5, 772--786.

\bibitem{IdczakRogowski}
D. Idczak and A. Rogowski,
On a generalization of Krasnoselskii's theorem,
\emph{Journal of the Australian Mathematical Society} \textbf{72} (2002), 389--394.

\bibitem{Jabri}
Y. Jabri,
\emph{The Mountain Pass Theorem: Variants, Generalizations and Some Applications},
Encyclopedia of Mathematics and its Applications, vol.~95,
Cambridge University Press, Cambridge, 2003.

\bibitem{Kigami1989}
J. Kigami,
A harmonic calculus on the Sierpi\'nski spaces,
\emph{Japan Journal of Applied Mathematics} \textbf{6} (1989), 259--290.

\bibitem{Kigami2001}
J. Kigami,
\emph{Analysis on Fractals},
Cambridge University Press, Cambridge, 2001.

\bibitem{MolicaBisciRadulescu2015}
G. Molica Bisci and V.~D. R\u{a}dulescu,
A characterization for elliptic problems on fractal sets,
\emph{Proceedings of the American Mathematical Society} \textbf{143} (2015), no.~7, 2959--2968.

\bibitem{SahuPriyadarshi2021}
A. Sahu and A. Priyadarshi,
A system of $p$-Laplacian equations on the Sierpi\'nski gasket,
\emph{Mediterranean Journal of Mathematics} \textbf{18} (2021), 92.

\bibitem{Strichartz2006}
R.~S. Strichartz,
\emph{Differential Equations on Fractals: A Tutorial},
Princeton University Press, Princeton, NJ, 2006.

\bibitem{strichartzwong2004}
R.~S. Strichartz and C. Wong,
The $p$-Laplacian on the Sierpi\'nski gasket,
\emph{Nonlinearity} \textbf{17} (2004), no.~2, 595--616.

\bibitem{Szulkin}
A. Szulkin,
Minimax principles for lower semicontinuous functions and applications to nonlinear boundary value problems,
\emph{Annales de l'Institut Henri Poincar\'e, Analyse Non Lin\'eaire} \textbf{3} (1986), 77--109.

\end{thebibliography}
\end{document}